\documentclass[12pt]{amsart}

\usepackage{amsmath,amsthm,amssymb}
\usepackage{amsfonts}
\usepackage{comment}
\usepackage{amsrefs}
\usepackage{hyperref}
\usepackage{mathtools}
\usepackage{tikz}
\usepackage{bm}
\usepackage[margin=1in]{geometry}

\usepackage{mathrsfs}
\usepackage{enumerate}
\usepackage[shortlabels]{enumitem}
\usepackage{todonotes}

\newtheorem{theorem}{Theorem}[section]

\newtheorem{lemma}[theorem]{Lemma}

\theoremstyle{definition}

\newtheorem*{conjecture*}{Conjecture}
\newtheorem*{proposition*}{Proposition}
\newtheorem*{lemma*}{Lemma}

\theoremstyle{remark}
\newtheorem*{remark}{Remark}

\numberwithin{equation}{section}

\newcommand{\CC}{\mathbb C}

\newcommand{\ZZ}{\mathbb Z}

\newcommand{\RR}{\mathbb R}

\newcommand{\on}{\operatorname}

\newcommand{\modd}{\on{mod}}
\makeatletter
\newcommand*{\defeq}{\mathrel{\rlap{%
                     \raisebox{0.3ex}{$\m@th\cdot$}}%
                     \raisebox{-0.3ex}{$\m@th\cdot$}}%
                     =}
\makeatother

\newcommand{\eps}{\varepsilon}
\newcommand\half{\frac12}
\newcommand{\real}{\on{Re}}

\newcommand{\sgn}{\on{sgn}}

\usepackage{mathrsfs}

\author[Yang Liu]{Yang Liu}
\address{Department of Mathematics, Massachusetts Institute of Technology, \mbox{Cambridge, MA 02139}}
\email{\href{mailto:yliu@mit.edu}{{\tt yliu97@mit.edu}}}

\author[Peter S. Park]{Peter S. Park}
\address{Department of Mathematics, Princeton University, \mbox{Princeton, NJ 08544}}
\email{\href{mailto:pspark@math.princeton.edu}{{\tt pspark@math.princeton.edu}}}

\author[Zhuo Qun Song]{Zhuo Qun Song}
\address{Department of Mathematics, Princeton University, \mbox{Princeton, NJ 08544}}
\email{\href{mailto:zsong@princeton.edu}{{\tt zsong@princeton.edu}}}

\begin{document}

\title[RH is True for Period Polynomials of Almost All Newforms]{The ``Riemann Hypothesis" is True for Period Polynomials of Almost All Newforms}
\date{\today}

\begin{abstract}
The period polynomial $r_f(z)$ for a weight $k \ge 3$ and level $N$ newform $f \in S_k(\Gamma_0(N),\chi)$ is the generating function for special values of $L(s,f)$. The functional equation for $L(s, f)$ induces a functional equation on $r_f(z)$. Jin, Ma, Ono, and Soundararajan proved that for all newforms $f$ of even weight $k \ge 4$ and trivial nebetypus, the ``Riemann Hypothesis'' holds for $r_f(z)$: that is, all roots of $r_f(z)$ lie on the circle of symmetry $|z| =1/\sqrt{N}$. We generalize their methods to prove that this phenomenon holds for all but possibly finitely many newforms $f$ of weight $k \ge 3$ with any nebentypus. We also show that the roots of $r_f(z)$ are equidistributed if $N$ or $k$ is sufficiently large. 
\end{abstract}
\maketitle
\section{Introduction and Statement of Results}

Let $f \in S_k(\Gamma_0(N),\chi)$ be a newform of weight $k$, level $N$, and nebentypus $\chi$. Associated to $f$ is an $L$-function $L(s, f)$, which can be normalized so that the completed $L$-function \begin{equation} 
\Lambda(s, f) := N^{s/2} \int_0^\infty f(iy)y^{s-1} dy = \left(\frac{\sqrt{N}}{2\pi}\right)^s \Gamma(s) L(s, f) 
\end{equation}
satisfies the functional equation \begin{equation} \label{fe} \Lambda(s, f) = \epsilon(f) \Lambda(k-s, \bar{f}) \end{equation} for some $\epsilon(f) \in \CC$ with $|\epsilon(f)| = 1.$

The \emph{period polynomial} associated to $f$ is the degree $k-2$ polynomial defined by
\begin{equation}\label{defnperiod}
r_f(z) \defeq \int_0^{i \infty} f(y)(y-z)^{k-2} dy.
\end{equation}
By the binomial theorem, we have
\begin{align*}
r_f(z) &= i^{k-1}N^{-\frac{k-1}{2}} \sum_{n = 0}^{k-2} \binom{k-2}{n}(\sqrt{N}iz)^n \Lambda(k-1-n, f), \\
&= -\frac{(k-2)!}{(2\pi i)^{k-1}}\sum_{n = 0}^{k-2} \frac{(2\pi iz)^n}{n!} L(k-1-n,f).
\end{align*} 
Thus, $r_f(z)$ is the generating function for the special values $L(1,f), L(2,f) \ldots, L(k-1,f)$ of the $L$-function associated to $f$. For background on period polynomials, we refer the reader to \cites{period1,period2,period3,period4,period5}.

When $k \ge 3$, the period polynomial $r_f(z)$ is nonconstant, so one can consider where the roots of $r_f(z)$  are located. To this end, we use  the functional equation \eqref{fe} to observe that
\[
\overline{r_f(z)} = -(\sqrt{N}i\overline z)^{k-2} \epsilon(f)^{-1}r_f\left(\frac{1}{N\overline{z}}\right).
\]
Thus, if $\rho$ is a root of $r_f(z)$, then $\frac{1}{N\bar\rho}$ is also a root. Much like the behavior of the nontrivial zeroes of $L(s, f)$ predicted by the Generalized Riemann Hypothesis, one can consider whether all the roots of $r_f(z)$ lie on the curve of symmetry of the roots: in this case, the circle $|\rho| = 1/\sqrt N$. It is natural to expect the following conjecture, which is supported by extensive numerical evidence.

\begin{conjecture*}[``Riemann Hypothesis" for period polynomials]
Let $f \in S_k(\Gamma_0(N),\chi)$ be a newform. Then, the roots of $r_f(z)$ all lie on the circle $|\rho| = 1/\sqrt N$.
\end{conjecture*}
El-Guindy and Raji~\cite{elg} proved this for Hecke eigenforms on $SL_2(\ZZ)$ with full level ($N=1$, for which the circle of symmetry is $|z| = 1$). They were inspired by the work of Conrey, Farmer, and \.{I}mamo\u{g}lu~\cite{farmer}, who showed an analogous result for the odd parts of these period polynomials, again with full level. 

Recent work by Jin, Ma, Ono, and Soundararajan~\cite{ono} proved the conjecture for all newforms of even weight $k \ge 4$ and trivial nebentypus. They also showed that the roots of $r_f(z)$ are equidistributed on the circle of symmetry for sufficiently large $N$ or $k$. Using similar methods, L{\"o}brich, Ma, and Thorner~\cite{thorner} proved an analogous result for polynomials generating special values of $L(s,\mathcal M)$ for a sufficiently well-behaved class of motives $\mathcal M$ with odd weight and even rank. 

In this paper, we generalize the methods of \cite{ono} to prove the conjecture for all but possibly finitely many newforms.
\begin{theorem}\label{thm:main}
The ``Riemann Hypothesis'' for period polynomials holds for all but possibly finitely many newforms with weight $k \ge 3$ and nontrivial nebentypus.
\end{theorem}

\begin{remark} Note that for $k < 3$, the period polynomial is a constant. Therefore, Theorem~\ref{thm:main} is essentially the best result for which one could hope, since an effective computation can check that Theorem~\ref{thm:main} also holds for the finitely many possible exceptions. We denote the set of these finitely many newforms as $\mathcal S$, which consists of the following:
\begin{enumerate}
\item For $k=5$, all newforms with level $N \le 10331$.
\item For $k \ge 6$, all newforms with level $N \le C(k)$, where $C(k)$ is a constant given by tables at the end of Section \ref{table1} and \ref{table2}.
\end{enumerate}
We know of no counterexamples to Theorem \ref{thm:main}.
\end{remark}

We also show that the roots of $r_f(z)$ are equidistributed on the circle of symmetry for sufficiently large $N$ or $k$. 
\begin{theorem}
\label{thm:equi}
Let $f \in S_k(\Gamma_0(N),\chi)$ be a newform of weight $k \ge 4$, level $N$, and nebentypus $\chi$ such that $f \notin \mathcal S$. Then, the following are true:
\begin{enumerate}[label={\upshape(\roman*)}]
\item Suppose that $k = 4$, and let $z_1, z_2$ denote the roots of $r_f(z)$. Then for any real $\eps > 0$,
\[
\arg z_1 -\arg z_2 \equiv \pi + O_\eps(N^{-\frac{1}{4} + \eps}) \pmod{2\pi},
\]
where the implied constant depends only on $\eps$ and is effectively computable.
\item Suppose that $k=5$. There exists $c_f \in \RR$ such that the arguments of the roots of $r_f(z)$ can be written as
\[
 c_f + \theta_{\ell} + O_{\eps}\left(\frac{1}{N^{\frac{1}{2} -\eps}}\right) \pmod{2\pi}, \hspace{40pt} 0 \le \ell \le 2,
\]
where $\theta_{\ell}$ denotes the unique solution $\modd$ $2\pi$ of
\[
\frac{k-2}{2}\theta_{\ell} - \frac{2\pi}{\sqrt N} \sin \theta_\ell = \ell \pi,
\]
and the implied constant depends only on $\eps$ and is effectively computable.

\item Suppose that $k >5$. There exists $c_f \in \RR$ such that the arguments of the roots of $r_f(z)$ can be written as
\[
c_f + \theta_{\ell} + O\left(\frac{1}{2^{k/2}\sqrt N}\right) \pmod{2\pi}, \hspace{40pt} 0 \le \ell \le k-3,
\]
Here, $\theta_{\ell}$ is the unique solution $\modd$ $2\pi$ to the equation
\[
\frac{k-2}{2}\theta_{\ell} - \frac{2\pi}{\sqrt N} \sin \theta_\ell = \ell \pi,
\]
and the implied constant is absolute and effectively computable.
\end{enumerate}
\end{theorem}
In Section~\ref{sec:prelims}, we introduce notation and lemmas that we will be using in our proof. In Section~\ref{sec:lowweight}, we will prove our main results for $k=3, 4,$ and $5$ using ad hoc arguments. For larger $k$, we prove Theorem~\ref{thm:main} in Section~\ref{sec:evenweight} (the case of $k$ even) and Section~\ref{sec:oddweight} (the case of $k$ odd), and we prove Theorem~\ref{thm:equi} in Section~\ref{sec:equi}. Finally, in Section~\ref{sec:numerical}, we detail our Sage computations suggesting that the roots of the period polynomial of the newform
\[
f(\tau) = q+10q^3+64q^4+74q^5+O(q^6) \in S_7\left(\Gamma_0(11),\left(\frac{-11}{\bullet}\right)\right)
\]
are all on the circle $|z|=1/\sqrt{11}$. This newform $f$ is in our finite set $\mathcal S$ of possible exceptions, which suggests that Theorem~\ref{thm:main} should be true even for newforms in $\mathcal S$.

\section{Preliminaries}\label{sec:prelims}
Throughout this section, we assume that $f \in S_k(\Gamma_0(N), \chi)$ is a newform of weight $k \ge 3$, level $N$, and arbitrary nebentypus $\chi.$ We note that the nebentypus character will be essentially invisible throughout our proof, other than the fact that it determines the level of $f$. We now define some notation related to $r_f(z)$ and prove lemmas about the values of $\Lambda(s,f)$ and $L(s,f)$ along the real line. The lemmas will be very similar in spirit to those proven in \cite{ono}.

Define $\delta$ to satisfy $\delta^2 = \epsilon(f)^{-1}.$ Now, define \begin{equation} \label{tpoly} t_f(z) = i^{-k+1}N^\frac{k-1}{2}z^{-\frac{k-2}{2}}\delta \cdot r_f\left(\frac{z}{i\sqrt{N}}\right) = \sum_{n = 0}^{k-2} \binom{k-2}{n}z^{n-\frac{k-2}{2}} \delta\Lambda(k-1-n, f), \end{equation} where $z^\frac{1}{2}$ denotes $r^\frac{1}{2}e^{\theta i/2}$ for $z = re^{\theta i}$ and $0 \le \theta < 2\pi.$ Using \eqref{fe}, one can compute
\[
\overline{t_f(z)} = t_f\left(\frac{1}{\bar{z}}\right).
\]
Therefore, if $t_f(z) = 0$, then $t_f\left(\frac{1}{\bar{z}}\right) = 0.$ Additionally, for $|z| = 1$, we also have $\overline{t_f(z)} = t_f(z)$, so $t_f(z)$ is real for $|z| = 1.$ Note that $t_f(z) = 0$ if and only if $r_f(\frac{z}{i\sqrt{N}}) = 0.$ Therefore, to prove Theorems \ref{thm:main} and \ref{thm:equi}, it suffices to show that all roots of $t_f(z)$ lie on the circle $|z| = 1$ and are equidistributed.

We will require the following monotonicity result.
\begin{lemma} \label{mono} We have \[ \Big|\Lambda\Big(\frac{k}{2}, f\Big)\Big| < \Big|\Lambda\Big(\frac{k}{2}+1, f\Big)\Big| < \dots < \Big|\Lambda\Big(\frac{k}{2}+j, f\Big)\Big| < \cdots \] Also, for all $0 < a < b$, \[ \Big| \Lambda\Big(\frac{k+1}{2} + a, f\Big) \Big| < \Big| \Lambda\Big(\frac{k+1}{2} + b, f\Big) \Big|. \] \end{lemma}

\begin{proof} As $\Lambda(s, f)$ is entire of order $1$, we apply the Hadamard factorization theorem to write
\[ \Lambda(s, f) = e^{A+Bs} \prod_\rho (1-s/\rho)e^{s/\rho}, \]
where the sum is taken over all roots $\rho$ of $\Lambda(s, f)$. By \cite[Proposition 5.7(3)]{ink}, we have that 
\[
\real(B) = -\sum_{\rho} \real\Big({\frac{1}{\rho}}\Big).
\]
Note that $\frac{k-1}{2} < \real{\rho} < \frac{k+1}{2}$. This implies that $|1 - \frac{s}{\rho}|$ is increasing for $s \ge \frac{k+1}{2}$ and $|1 - \frac{k/2}{\rho} | < |1 - \frac{k/2+1}{\rho}|$, from which the lemma follows.
\end{proof}

We also prove a useful inequality on ratios of $L$-function values.
\begin{lemma} \label{lratio} For all $0 < a < b$, we have \[ \left|\frac{L(\frac{k+1}{2}+a, f)}{L(\frac{k+1}{2}+b, f)} - 1\right| \le \frac{\zeta(1+a)^2}{\zeta(1+b)^2}-1. \] \end{lemma}

\begin{proof}
We have that
\[ \left|\frac{L(\frac{k+1}{2}+a, f)}{L(\frac{k+1}{2}+b, f)} - 1\right| = \left| \exp\left(-\int_a^b \frac{L'(\frac{k+1}{2}+s, f)}{L(\frac{k+1}{2}+s, f)} ds \right)-1 \right|. \]
If we express \[ -\frac{L'(s, f)}{L(s, f)} = \sum \frac{\Lambda_f(n)}{n^s}, \] then Deligne's bound on the eigenvalues of the Hecke operators on $S_k(\Gamma_0(N),\chi)$ states that 
\begin{equation}\label{deligne23}
|\Lambda_f(n)| \le 2n^\frac{k-1}{2}\Lambda(n)
\end{equation}
where $\Lambda(n)$ denotes the von Mangoldt function; for a reference, see \cite[Theorem 2.32]{onobook}. Therefore, we have that
\begin{equation}\label{deligne}\left| \int_a^b \frac{L'(\frac{k+1}{2} + s, f)}{L(\frac{k+1}{2} + s, f)} \right| \le -2 \int_a^b \frac{\zeta'(1+s)}{\zeta(1+s)} ds = 2 \log \frac{\zeta(1+a)}{\zeta(1+b)}. \end{equation} Now the lemma follows from the inequality $|e^x-1| \le e^{|x|}-1.$
\end{proof}

Finally, we show a lemma that serve as our main means of proving Theorem \ref{thm:main} for period polynomials.

\begin{lemma}
\label{ivt}
Let $\sgn(r)$ equal $-1$ for negative real numbers $r$, $1$ for positive real numbers $r$, and $0$ for $r = 0.$ If there exist real numbers $0 \le \theta_1 < \theta_2 < \dots < \theta_k < 2\pi$ such that either \[ \sgn(t_f(e^{i\theta_j})) = (-1)^j \text{ for all } 1 \le j \le k, \] or \[ \sgn(t_f(e^{i\theta_j})) = (-1)^{j+1} \text{ for all } 1 \le j \le k, \] then all solutions to $t_f(z) = 0$ satisfy $|z| = 1.$
\end{lemma}
\begin{proof}
First, $t_f(z)$ is real for $|z| = 1$, so $\sgn(t_f(e^{i\theta_j}))$ is well defined. Now, by the Intermediate Value Theorem, there exist $\theta \in (\theta_j, \theta_{j+1})$ such that $t_f(e^{i\theta}) = 0$ for all $1 \le j \le k-1.$ This gives us $k-1$ roots of $t_f(z)$ that lie on $|z| = 1.$ When $k$ is even, we also get a root in the range $(\theta_k, \theta_1+2\pi)$ by the Intermediate Value Theorem. When $k$ is odd, we may redefine the square root in order to move the discontinuity into an interval outside of $[\theta_k, \theta_1 + 2\pi]$. This would only affect the sign of $t_f(z)$. By the intermediate value theorem, this shows the existence of a zero with argument in the range $[\theta_k, \theta_1 + 2\pi]$ as desired. As $t_f(z) = 0$ for at most $k$ values of $z$, the above argument shows that we have found all of them.
\end{proof}

\section{Proof for Weights 3, 4, and 5}\label{sec:lowweight}
Here we prove Theorem \ref{thm:main} and \ref{thm:equi} for $k = 3, 4,$ and $5.$ 

\subsection{The weight 3 case}

For $k = 3,$ (\ref{tpoly}) gives that \[ t_f(z) = \delta z^{-\frac{1}{2}}(\Lambda(2, f) + z\Lambda(1, f)). \] By \eqref{fe}, we know that \[ |\Lambda(2, f)| = |\Lambda(1, f)|, \] so the root of $t_f(z)$ lies on the unit circle.

\subsection{The weight 4 case}

For $k = 4$, we have \[ t_f(z) = \delta(z^{-1}\Lambda(3, f) + 2\Lambda(2, f) + z\Lambda(1, f)). \] Now, note that for $|z| = 1$, it follows that \[ \frac{1}{2}t_f(z) = \real(\delta(\Lambda(2, f) + z\Lambda(1, f))) = \real(\delta\Lambda(2, f)) + \real(\delta z\Lambda(1, f)).\] By Lemma \ref{mono}, we have that \[ |\real(\delta\Lambda(2, f))| \le |\Lambda(2, f)| < |\Lambda(3, f)| = |\Lambda(1, f)|, \] so there exist $2$ values of $z$ with $|z| = 1$ such that \[ \real(\delta z\Lambda(1, f)) = -\real(\delta\Lambda(2, f)), \] as desired.

In order to prove Theorem~\ref{thm:equi}(iii), we need to bound $|\Lambda(2, f)| / |\Lambda(1, f)|$. First, note that
\[
|\Lambda(1, f)| = |\Lambda(3, f)| \gg N^{\frac{3}{2}}.
\]
In order to bound $|\Lambda(2, f)|$, we appeal to the Phragm{\'e}n-Lindel{\"o}f Principle; specifically, see \cite[Lemma 5.2, Theorem 5.53]{ink} and apply \eqref{deligne23}. This allows to obtain for any $\epsilon > 0$
\[
|\Lambda(2, f)| \le \max_{t \in \mathbb{R}} |\Lambda(5/2 + \epsilon + it, f)| = \max_{t \in \mathbb{R}} N^{\frac{5}{4} +\half\epsilon}|L(5/2 + \epsilon + it, f)| \ll_{\epsilon} N^{\frac{5}{4}+\epsilon},
\]

Thus, we have that
\[
\frac{|\Lambda(2, f)|}{|\Lambda(1, f)|} \ll N^{-\frac{1}{4}+\epsilon}
\] %add convexity bound
and the values of $z$ satisfying $t_f(z) = 0$ satisfy
\[
\arg z = \pm \frac{\pi}{2} + \arg(\overline{\delta \Lambda(1, f)}) + O(N^{-\frac{1}{4}+\epsilon})
\]
\subsection{The weight 5 case}

For $k = 5$, we have \[ t_f(z) = \delta(z^{-\frac{3}{2}}\Lambda(4, f) + 3z^{-\frac{1}{2}}\Lambda(3, f) + 3z^\frac{1}{2}\Lambda(2, f) + z^\frac{3}{2}\Lambda(1, f)). \]
Once again, for $|z| = 1$, we have
\[
\frac{1}{2}t_f(z) = \real(\delta(3z^\frac{1}{2}\Lambda(2, f) + z^\frac{3}{2}\Lambda(1, f))) = \real(3\delta z^\frac{1}{2}\Lambda(2, f)) + \real(\delta z^\frac{3}{2}\Lambda(1, f)).
\]
There exist three reals $0 \le \theta_1 < \theta_2 < \theta_3 < 2\pi$ such that $|\real(\delta (e^{i\theta_j})^\frac{3}{2} \Lambda(1, f))| = |\Lambda(1, f)|$ for $1 \le j \le 3$, and $\real(\delta (e^{i\theta_j})^\frac{3}{2} \Lambda(1, f))$ alternates in sign. Thus, by Lemma \ref{ivt}, we are done if we are able to show that
\[ |\real(3\delta z^\frac{1}{2}\Lambda(2, f))| < |\Lambda(1, f)|,
\]
which is equivalent to proving
\[
\frac{|\Lambda(3, f)|}{|\Lambda(4, f)|} < \frac{1}{3}.
\]
Let $0 < \epsilon < 1.$ By Lemmas \ref{mono} and \ref{lratio}, it follows that
\[ \frac{|\Lambda(3, f)|}{|\Lambda(4, f)|} \le \frac{|\Lambda(3+\epsilon, f)|}{|\Lambda(4, f)|} = \frac{|L(3+\epsilon, f)|}{|L(4, f)|} \left(\frac{2\pi}{\sqrt{N}}\right)^{1-\epsilon} \frac{\Gamma(3+\epsilon)}{\Gamma(4)} \le \frac{\zeta(1+\epsilon)^2}{\zeta(2)^2} \left(\frac{2\pi}{\sqrt{N}}\right)^{1-\epsilon} \frac{\Gamma(3+\epsilon)}{\Gamma(4)}.
\]
Choosing $\epsilon = 2/5$, the last expression is less than $\frac{1}{3}$ for $N \ge 10332$, which completes the proof for $k = 5.$

To show the desired equidistribution property, define $\theta_1$ and $\theta_2$ as above. Now let $\theta_{\pm} = (\theta_1 + \theta_2)/2 \pm \eps$, for $\eps > 0$ to be chosen later. Then, we see that
\[
|\real (\delta (e^{\theta_{\pm} i})^{\frac{3}{2}} \Lambda(1, f)| = |\Lambda(1, f)| \sin \Big(\frac{3}{2}\eps\Big)
\]
with the sign of $\real (\delta (e^{i\theta_{\pm}})^{\frac{3}{2}} \Lambda(1, f)$ being different for $\eps > 0$ and $\eps < 0$.
If we can show that
\[
|\real (3\delta z^{\half} \Lambda(2, f)| \le \Big|\sin \Big(\frac{3}{2}\eps\Big)\Big|\cdot|\Lambda(1, f)|
\]
then Lemma \ref{ivt} will show that the root has an argument lying between $\theta_{-}$ and $\theta_{+}$. By the bounding above, we only require
\[
\frac{\zeta(1+\epsilon)^2}{\zeta(2)^2} \left(\frac{2\pi}{\sqrt{N}}\right)^{1-\epsilon} \frac{\Gamma(3+\epsilon)}{\Gamma(4)} < \frac{1}{3} \sin \Big(\frac{3}{2}\eps\Big).
\]
Choosing $\eps = O(N^{-\half + \epsilon})$ suffices.

\section{Proof for Remaining Even Weights}\label{sec:evenweight}
In this section, we will show Theorem \ref{thm:main} for all even weights $k \ge 6.$ Throughout the section, we will restrict our attention to those $z$ such that $|z| = 1.$ For simplicity, let $m = \frac{k-2}{2}$, and define \begin{align} P_f(z) &= \frac{1}{2} \binom{2m}{m} \Lambda(m+1, f) + \sum_{n = 0}^{m-1} \binom{k-2}{n} z^{m-n} \delta \Lambda(n+1, f) \nonumber \\ \label{pfeven} &= \frac{1}{2} \binom{2m}{m} \Lambda(m+1, f) + (2m)!\left(\frac{\sqrt{N}}{2\pi}\right)^{2m+1} \delta^{-1} z^m \sum_{n = 0}^{m-1} \frac{1}{n!}\left(\frac{2\pi}{z\sqrt{N}}\right)^n \overline{L(2m+1-n, f)}. \end{align}
This satisfies
\[
t_f(z) = P_f(z) + \overline{P_f(z)} = 2 \real(P_f(z)).
\]

Next, define
\begin{align*} Q_f(z) &= \frac{1}{(2m)!} \left(\frac{2\pi}{\sqrt{N}}\right)^{2m+1}P_f(z) \\
&= \frac{1}{2(m!)^2}\left(\frac{2\pi}{\sqrt{N}}\right)^{2m+1} \Lambda(m+1, f) + \delta^{-1}z^m \sum_{n = 0}^{m-1}\frac{1}{n!}\left(\frac{2\pi}{z\sqrt{N}}\right)^n \overline{L(2m+1-n, f)}. \end{align*}
Note that $\sgn(Q_f(z)) = \sgn(P_f(z)).$ As in \cite{ono}, rewrite
\[
Q_f(z) = \delta^{-1} \overline{L(2m+1, f)} z^m \exp\left(\frac{2\pi}{z\sqrt{N}}\right) + S_1(z) + S_2(z) + S_3(z),
\]
where we define
\begin{align*}
&S_1(z) = \delta^{-1} \overline{L(2m+1, f)} z^m \sum_{n = 0}^{m-1} \frac{1}{n!} \left(\frac{2\pi}{z\sqrt{N}}\right)^n \left(\overline{\left(\frac{L(2m+1-n, f)}{L(2m+1, f)}\right)}-1 \right) \\
&S_2(z) = -\delta^{-1} \overline{L(2m+1, f)} z^m \sum_{n \ge m} \frac{1}{n!} \left(\frac{2\pi}{z\sqrt{N}}\right)^n \\
&S_3(z) = \frac{1}{2(m!)^2} \left(\frac{2\pi}{\sqrt{N}}\right)^{2m+1} \Lambda(m+1, f).
\end{align*}
For $z = e^{i\theta}$, note that
\begin{equation} \label{defC}
\arg\left(\delta^{-1} \overline{L(2m+1, f)} z^m \exp\left(\frac{2\pi}{z\sqrt{N}}\right)\right) = C + m\theta - \frac{2\pi \sin\theta}{\sqrt{N}},
\end{equation}
where $C$ is a fixed constant depending on $\delta$ and $L(2m+1, f).$ Therefore, we can pick $k$ values of $z$ on the circle $|z| = 1$ such that the previous expression has argument $\ell\pi$ for integers $\ell.$ The value of $Q_f(z)$ at these points have alternating positive and negative real part with magnitude at least $|L(2m+1, f)|\exp\left(-\frac{2\pi}{\sqrt{N}}\right)$. By Lemma \ref{ivt}, it suffices to show that
\[
|S_1(z)| + |S_2(z)| + |S_3(z)| < |L(2m+1, f)|\exp\left(-\frac{2\pi}{\sqrt{N}}\right).
\]
To bound $S_1(z)$, we use Lemma \ref{lratio} in the form $\Big|\frac{L(2m+1-n, f)}{L(2m+1, f)} \Big| \le \zeta(\frac{1}{2}+m-n)^2-1.$ This gives \[ |S_1(z)| \le |L(2m+1, f)| \sum_{n = 1}^{m-1} \frac{1}{n!} \left(\frac{2\pi}{\sqrt{N}}\right)^n (\zeta(1/2+m-n)^2-1). \] For the term $n = m-1$ in the above expression, we use the bound $\zeta(3/2)^2-1 \le 35/6$. For $0 \le n \le m-2$, note that $2^x(\zeta(1/2+x)^2-1)$ is decreasing for $x \ge 2.$ Therefore, for $0 \le n \le m-2$, we find that \[ \zeta(1/2 + m-n)^2-1 \le 2^{n-m}\cdot 4(\zeta(5/2)^2-1) \le \frac{16}{5}2^{n-m}. \] Now, we combine the above estimates with $S_2(z)$ to obtain \begin{align} \frac{|S_1(z)| + |S_2(z)|}{|L(2m+1, f)|} &\le \frac{16}{5}\sum_{n = 1}^{m-1} \frac{1}{n!} \left(\frac{2\pi}{\sqrt{N}}\right)^n \frac{2^n}{2^m} + \frac{17}{4} \frac{1}{(m-1)!} \left(\frac{2\pi}{\sqrt{N}}\right)^{m-1} + \sum_{n \ge m} \frac{1}{n!} \left(\frac{2\pi}{\sqrt{N}}\right)^n \nonumber \\ \label{s12} &\le \frac{16}{5}2^{-m} \left(\exp\left(\frac{4\pi}{\sqrt{N}}\right) - 1\right) + \frac{17}{4} \frac{1}{(m-1)!} \left(\frac{2\pi}{\sqrt{N}}\right)^{m-1}. \end{align}
To finish, we estimate $|S_3(z)|$ using Lemma \ref{mono} and then \ref{lratio}. \begin{align} |S_3(z)| &\le \frac{1}{2(m!)^2} \left(\frac{2\pi}{\sqrt{N}}\right)^{2m+1} |\Lambda(m+1, f)| \le \frac{1}{2(m!)^2} \left(\frac{2\pi}{\sqrt{N}}\right)^{2m+1} |\Lambda(m+2, f)| \nonumber \\ &\le \frac{m+1}{2m!} \left(\frac{2\pi}{\sqrt{N}}\right)^{m-1} |L(m+2, f)| \le \frac{m+1}{2m!} \left(\frac{2\pi}{\sqrt{N}}\right)^{m-1} |L(2m+1, f)| \zeta(3/2)^2 \nonumber \\ \label{s3} &\le \frac{7}{2} \frac{m+1}{m!} \left(\frac{2\pi}{\sqrt{N}}\right)^{m-1} |L(2m+1, f)|. \end{align}
By using (\ref{s12}) and (\ref{s3}), it suffices to verify \begin{equation} \label{finaleven} \frac{16}{5}2^{-m} \left(\exp\left(\frac{4\pi}{\sqrt{N}}\right) - 1\right) + \frac{17}{4} \frac{1}{(m-1)!} \left(\frac{2\pi}{\sqrt{N}}\right)^{m-1} + \frac{7}{2} \frac{m+1}{m!} \left(\frac{2\pi}{\sqrt{N}}\right)^{m-1} < \exp\left(-\frac{2\pi}{\sqrt{N}}\right). \end{equation}

For each value of $m$ in the first row on the following table, the value $N(m)$ is such that inequality \eqref{finaleven} holds for all $N\ge N(m)$. Note that the case $m = 1$ was done in Section \ref{sec:lowweight}.

\begin{center} \label{table1} 
\begin{tabular} { |c || c | c |c|c|c|c|c|c|c|c|c|c|c|c|c|c|c|c|c|c|}
\hline
   $m$&  29& 21& 18& 16& 14 & 13& 12& 11& 10 & 9&8&7&6&5&4&3&2 \\
\hline
  $N(m)$& 1&  2&  3&  4&  5& 6& 7& 9& 11&14& 19& 27&41&69&142&433&5875 \\
  \hline
  \end{tabular}
  \end{center} 
Therefore, for all $m \ge 29$, $N(m) = 1$. This completes our proof of Theorem \ref{thm:main} for $k$ even.

\section{Proof for Remaining Odd Weights}\label{sec:oddweight}

In this section, we will show Theorem \ref{thm:main} for all odd weights $k \ge 7.$ As in the above section, we will restrict our attention to those $z$ such that $|z| = 1.$ For simplicity, let $m = \frac{k-3}{2}$, and define \begin{align} P_f(z) &= \sum_{n = 0}^{m} \binom{k-2}{n} z^{m-n+\frac{1}{2}} \delta \Lambda(n+1, f) \nonumber \\ \label{pfodd} &= (2m+1)! \left(\frac{\sqrt{N}}{2\pi}\right)^{2m+2}\delta^{-1} z^{m+\frac{1}{2}} \sum_{n = 0}^m \frac{1}{n!}\left(\frac{2\pi}{z\sqrt{N}}\right)^n \overline{L(2m+2-n, f)}, \end{align} so $t_f(z) = 2\real(P_f(z))$ As in the above section, define
\begin{align*}
Q_f(z) &= \frac{1}{(2m+1)!} \left(\frac{2\pi}{\sqrt{N}}\right)^{2m+2} P_f(z) = \delta^{-1} z^{m+\frac{1}{2}} \sum_{n = 0}^m \frac{1}{n!}\left(\frac{2\pi}{z\sqrt{N}}\right)^n \overline{L(2m+2-n, f)} \\
&= \overline{L(2m+2, f)}\delta^{-1}z^{m+\frac{1}{2}} \exp\left(\frac{2\pi}{z\sqrt{N}}\right) + S_1(z) + S_2(z) + S_3(z),
\end{align*}
where $S_1(z), S_2(z)$, and $S_3(z)$ are defined as follows.
\[
S_1(z) = \overline{L(2m+2, f)}\delta^{-1}z^{m+\frac{1}{2}} \sum_{n = 0}^{m-1} \frac{1}{n!}\left(\frac{2\pi}{z\sqrt{N}}\right)^n\left(\overline{\left(\frac{L(2m+2-n, f)}{L(2m+2, f)}\right)} - 1\right),
\]
\[
S_2(z) = -\overline{L(2m+2, f)}\delta^{-1}z^{m+\frac{1}{2}} \sum_{n \ge m} \frac{1}{n!} \left(\frac{2\pi}{z\sqrt{N}}\right)^n,
\]
\[
S_3(z) = \delta^{-1}z^{m+\frac{1}{2}} \frac{1}{m!} \left(\frac{2\pi}{z\sqrt{N}}\right)^m \overline{L(m+2, f)}.
\]
As in Section \ref{sec:evenweight}, it suffices to show that
\[
|S_1(z)| + |S_2(z)| + |S_3(z)| < |L(2m+2, f)|\exp\left(-\frac{2\pi}{\sqrt{N}}\right).
\]
The proof of this will proceed in a very similar way to that of the above section. Note that the function $2^x(\zeta(1+x)^2-1)$ is decreasing for $x \ge 1.$ By Lemma \ref{lratio}, for $0 \le n \le m-1$, we can bound \[ \Big|\frac{L(2m+2-n, f)}{L(2m+2, f)} - 1 \Big| \le \zeta(1+m-n)^2-1 \le 2^{n-m} \cdot 2(\zeta(2)^2-1) \le 2^{n-m} \cdot \frac{7}{2}. \] By Lemma \ref{lratio}, we have \begin{equation} \label{s12odd} \frac{|S_1(z)| + |S_2(z)|}{|L(2m+2, f)|} \le \frac{7}{2}2^{-m} \sum_{n = 1}^{m-1} \frac{1}{n!} \left(\frac{4\pi}{\sqrt{N}}\right)^n + \sum_{n \ge m} \frac{1}{n!} \left(\frac{2\pi}{\sqrt{N}}\right)^n \frac{2^n}{2^m} \le \frac{7}{2}2^{-m} \left(\exp\left(\frac{4\pi}{\sqrt{N}}\right) - 1\right). \end{equation} Now we use Lemma \ref{mono} to bound $|L(m+2, f)|$. \[ |L(m+2, f)| \le \frac{1}{(m+1)!}\left(\frac{2\pi}{\sqrt{N}}\right)^{m+2} |\Lambda(m+2, f)| \le (m+2) \left(\frac{2\pi}{\sqrt{N}}\right)^{-1} |L(m+3, f)|. \] Therefore, we have that \begin{equation} \label{s3odd} |S_3(z)| \le \frac{m+2}{m!}\left(\frac{2\pi}{\sqrt{N}}\right)^{m-1}|L(m+3, f)| \le \frac{m+2}{m!}\left(\frac{2\pi}{\sqrt{N}}\right)^{m-1} \zeta(2)^2|L(2m+2, f)|, \end{equation} after applying Lemma $\ref{lratio}.$ Finally, by using (\ref{s12odd}) and (\ref{s3odd}), it suffices to show that \begin{equation} \label{finalodd} \frac{7}{2}2^{-m} \left(\exp\left(\frac{4\pi}{\sqrt{N}}\right) - 1\right) + \frac{m+2}{m!}\left(\frac{2\pi}{\sqrt{N}}\right)^{m-1} \zeta(2)^2 < \exp\left(-\frac{2\pi}{\sqrt{N}}\right). \end{equation}

For each value of $m$ in the first row on the following table, the value $N(m)$ is such that inequality \eqref{finalodd} holds for all $N\ge N(m)$. Note that the cases $m = 0, 1$ was done in Section \ref{sec:lowweight}.

\begin{center} \label{table2}
\begin{tabular} 
{ |c || c | c |c|c|c|c|c|c|c|c|c|c|c|c|c|c|c|c|c|c|}
\hline
   $m$&  31& 23& 19& 17& 16 & 15& 14& 13& 12 & 11 & 10 & 9&8&7&6&5&4&3&2 \\
\hline
  $N(m)$& 1&  2&  3&  4&  5& 6& 7& 8& 10&13& 16& 22&31&47&76&137&285&766&5258 \\
  \hline
  \end{tabular}
  \end{center} 

Therefore, for $m \ge 31$, $N(m) = 1$. This completes the proof of Theorem \ref{thm:main} for $k$ odd.

\section{Equidistribution of Roots for Large Weights}\label{sec:equi}

Let $k \ge 6.$ For even $k$, set $m = \frac{k-2}{2}$. Then, the arguments in the previous sections show that for $z = e^{i\theta}$, then \[ Q_f(z) = |L(2m+1, f)| \left(\exp\left(i(m\theta+C) + \frac{2\pi}{\sqrt{N}}e^{-i\theta}\right) + O\Big(\frac{1}{2^m \sqrt{N}}\Big) \right) \] for some real constant $C$ is defined in \eqref{defC}. Therefore, \[ \real(Q_f(z)) = |L(2m+1, f)|\left(\exp\left(\frac{2\pi}{\sqrt{N}} \cos \theta\right) \cos\left(m\theta+C - \frac{2\pi}{\sqrt{N}} \sin \theta \right) + O\Big(\frac{1}{2^m \sqrt{N}}\Big) \right). \] Consider the $\theta_\ell$ such that $m\theta_\ell+C - \frac{2\pi}{\sqrt{N}} \sin \theta_\ell = \frac{\pi}{2} + \ell \pi.$ Then it is simple to verify that for some constant $D$, the two values $\theta_\ell \pm \frac{D}{2^m \sqrt{N}}$, $\real(Q_f(z))$ has different signs. This completes the proof for even $k$.

For odd $k$, set $m = \frac{k-3}{2}.$ Then, the arguments in the previous section show that for $z = e^{\theta i}$,
\[
Q_f(z) = |L(2m+2, f)| \left(\exp\left(i\Big(\Big(m+\frac{1}{2}\Big)\theta+C\Big) + \frac{2\pi}{\sqrt{N}}e^{-\theta i}\right) + O\Big(\frac{1}{2^m \sqrt{N}}\Big) \right).
\]
Therefore, it follows that
\[
\real(Q_f(z)) = |L(2m+2, f)| \left(\exp\left(\frac{2\pi}{\sqrt{N}} \cos \theta\right)\cos\left(\Big(m+\frac{1}{2}\Big)\theta + C - \frac{2\pi}{\sqrt{N}} \sin \theta\right) + O\left(\frac{1}{2^m \sqrt{N}}\right) \right).
\]
Now, consider the values $\theta_\ell$ such that $(m+1/2)\theta_\ell + C - \frac{2\pi}{\sqrt{N}} \sin \theta_\ell = \frac{\pi}{2} + \ell\pi.$ Once again, one can verify that for the values $\theta_\ell \pm \frac{D}{2^m \sqrt{N}},$ $\real(Q_f(z))$ has opposite signs. This completes the proof of Theorem \ref{thm:equi}.

\section{A Numerical Example}\label{sec:numerical}
Consider the newform $f \in S_7\left(\Gamma_0(11),\left(\frac{-11}{\bullet}\right)\right)$ whose $q$-series is given by
\[
q+10q^3+64q^4+74q^5+O(q^6).
\]
All the coefficients of $f$ are real, and we have $\epsilon(f) = 1$. In light of the functional equation $L(s,f)=L(k-1-s,f)$, we can use Sage to compute the critical values of $L(s,f)$ and thereby obtain $r_f(z)$. We calculate that the roots of $r_f(z)$ are  
\[
z_1 \approx -0.294570496142963 - 0.0643219535709181 i,
\]
\[
z_2 \approx -0.204098252273756 + 0.221930156418385i,
\]
\[
z_3 \approx 0.301511344577764 i,
\]
\[
z_4 \approx 0.204098252273756 + 0.221930156418385i,
\]
and
\[
z_5 \approx 0.294570496142963 - 0.0643219535709181i.
\]
All five roots have absolute value $\approx 0.301511344577764 \approx 1/\sqrt{11}$, as expected given the statement of Theorem~\ref{thm:main}.

\section*{Acknowledgments}

\noindent This research was supervised by Ken Ono at the Emory University Mathematics REU and was supported by the National Science Foundation (grant number DMS-1557960). We would like to thank Ken Ono and Jesse Thorner for offering their advice and guidance and for providing many helpful discussions and valuable suggestions on the paper. 

\bibliographystyle{amsxport}
\bibliography{biblio}

\end{document}